\documentclass[11pt]{article}
\usepackage{amsmath,amssymb,amsthm,amscd}

\newtheorem{theorem}{Theorem}[section]

\newtheorem{proposition}[theorem]{Proposition}

\newtheorem{conjecture}[theorem]{Conjecture}

\theoremstyle{definition}
\newtheorem{definition}[theorem]{Definition}
\newtheorem{example}[theorem]{Example}
\newtheorem{remark}[theorem]{Remark}

\numberwithin{equation}{section}

\renewcommand{\phi}{\varphi}

\newcommand{\ep}{\varepsilon}

\newcommand{\A}{\operatorname{A}}
\renewcommand{\S}{\operatorname{S}}

\newcommand{\Aff}{\operatorname{Aff}}
\newcommand{\Coker}{\operatorname{Coker}}
\newcommand{\Homeo}{\operatorname{Homeo}}
\newcommand{\id}{\operatorname{id}}
\newcommand{\Ima}{\operatorname{Im}}
\newcommand{\Inf}{\operatorname{Inf}}
\newcommand{\Ker}{\operatorname{Ker}}
\newcommand{\sgn}{\operatorname{sgn}}
\newcommand{\Tor}{\operatorname{Tor}}

\newcommand{\ab}{\mathrm{ab}}
\newcommand{\K}{\mathcal{K}}
\newcommand{\G}{\mathcal{G}}
\renewcommand{\H}{\mathcal{H}}

\newcommand{\N}{\mathbb{N}}
\newcommand{\Z}{\mathbb{Z}}
\newcommand{\Q}{\mathbb{Q}}
\newcommand{\R}{\mathbb{R}}

\title{Topological full groups of \'etale groupoids}
\author{Hiroki Matui\\
Graduate School of Science \\
Chiba University \\
Inage-ku, Chiba 263-8522, Japan}
\date{}

\begin{document}
\maketitle

\begin{abstract}
This is a survey of the recent development of the study of 
topological full groups of \'etale groupoids on the Cantor set. 
\'Etale groupoids arise from dynamical systems, 
e.g. actions of countable discrete groups, equivalence relations. 
Minimal $\Z$-actions, minimal $\Z^N$-actions and 
one-sided shifts of finite type are basic examples. 
We are interested in algebraic, geometric and analytic properties 
of topological full groups. 
More concretely, 
we discuss simplicity of commutator subgroups, abelianization, 
finite generation, cohomological finiteness properties, 
amenability, the Haagerup property, and so on. 
Homology groups of \'etale groupoids, 
groupoid $C^*$-algebras and their $K$-groups are also investigated. 
\end{abstract}

\section{Introduction}

We discuss various properties of topological full groups of 
topological dynamical systems on Cantor sets. 
The study of full groups in the setting of topological dynamics 
was initiated by T. Giordano, I. F. Putnam and C. F. Skau \cite{GPS99Israel}. 
For a minimal action $\phi:\Z\curvearrowright X$ on a Cantor set $X$, 
they defined several types of full groups and showed that 
these groups completely determine the orbit equivalence class, 
the strong orbit equivalence class and the flip conjugacy class of $\phi$, 
respectively. 

The notion of topological full groups was later generalized to the setting of 
essentially principal \'etale groupoids $\G$ on Cantor sets in \cite{M12PLMS}. 
\'Etale groupoids (called $r$-discrete groupoids in \cite{Rtext}) provide us 
a natural framework for unified treatment of 
various topological dynamical systems. 
The topological full group $[[\G]]$ of $\G$ is 
a subgroup of $\Homeo(\G^{(0)})$ consisting of 
all homeomorphisms of $\G^{(0)}$ 
whose graph is `contained' in the groupoid $\G$ as a compact open subset 
(see Definition \ref{tfg/def}). 
From an action $\phi$ of a discrete group $\Gamma$ on a Cantor set $X$, 
we can construct the \'etale groupoid $\G_\phi$, 
which is called the transformation groupoid (see Example \ref{transf/def}). 
The topological full group $[[\G_\phi]]$ of $\G_\phi$ is 
the group of $\alpha\in\Homeo(X)$ 
for which there exists a continuous map $c:X\to\Gamma$ such that 
$\alpha(x)=\phi^{c(x)}(x)$ for all $x\in X$. 
Many other examples of \'etale groupoids and topological full groups 
will be provided in later sections. 

One of the most fundamental result for topological full groups is 
the isomorphism theorem (Theorem \ref{isomorphism}), 
which says that $\G_1$ is isomorphic to $\G_2$ if and only if 
$[[\G_1]]$ is isomorphic to $[[\G_2]]$. 
In general, it is often difficult to distinguish two discrete groups. 
But, the \'etale groupoids have rich information 
about the topological dynamical systems, 
and so the isomorphism theorem helps us 
to determine the isomorphism class of the topological full groups. 

The homology groups $H_n(\G)$ for $n\geq0$ are defined 
for \'etale groupoids $\G$ (see Definition \ref{homology}). 
When $\G$ is a transformation groupoid $\G_\phi$, 
the homology $H_n(\G_\phi)$ agrees with the group homology 
(Example \ref{ExofH} (2)). 
In many examples, we can check that 
the homology groups `coincide' with the $K$-groups of $C^*_r(\G)$. 
Thus, 
we have isomorphisms $\bigoplus_nH_{2n+i}(\G)\cong K_i(C^*_r(\G))$ 
for $i=0,1$. 
This phenomenon is formulated as the HK conjecture (Conjecture \ref{HK}). 

In many cases, it is known that 
the commutator subgroup $D([[\G]])$ of $[[\G]]$ becomes simple 
(Theorem \ref{almstfnt/thm} (1), Theorem \ref{pi/thm} (1)). 
So, 
it is natural to consider the abelianization $[[\G]]_\ab=[[\G]]/D([[\G]])$. 
It turns out that 
the abelian group $[[\G]]_\ab$ is closely related to 
the homology groups of $\G$. 
This relation is formulated as the AH conjecture (Conjecture \ref{AH}). 

In addition to these two conjectures, 
we are interested in several properties of $[[\G]]$. 
In \cite{M06IJM}, it was shown that $D([[\G_\phi]])$ is finitely generated 
if $\phi:\Z\curvearrowright X$ is a minimal subshift 
(see Theorem \ref{Z/finite} (1)). 
In \cite{M15crelle}, it was shown that, for any SFT groupoid $\G_A$ 
(see Example \ref{SFT/def}), $[[\G_A]]$ is of type F$_\infty$ and 
$D([[\G_A]])$ is finitely generated (Theorem \ref{SFT/finite}). 
Such finiteness conditions of topological full groups are important problems. 
In \cite{JM13Ann}, it was shown that, 
for any minimal action $\phi:\Z\curvearrowright X$, 
$[[\G_\phi]]$ is amenable (see Theorem \ref{Z/amenable}). 
In \cite{M15crelle}, it was shown that 
$[[\G_A]]$ has the Haagerup property for any SFT groupoid $\G_A$ 
(Theorem \ref{SFT/Haagerup}). 
Such analytic properties of topological full groups are also our main concern. 

I would like to thank the organisers of the 2015 Abel Symposium 
for their kind invitation to this marvellous conference.

\section{Preliminaries}

\subsection{\'Etale groupoids}

The cardinality of a set $A$ is written $\#A$ and 
the characteristic function of $A$ is written $1_A$. 
The finite cyclic group of order $n$ is denoted 
by $\Z_n=\{\bar{r}\mid r=1,2,\dots,n\}$. 
We say that a subset of a topological space is clopen 
if it is both closed and open. 
A topological space is said to be totally disconnected 
if its topology is generated by clopen subsets. 
By a Cantor set, 
we mean a compact, metrizable, totally disconnected space 
with no isolated points. 
It is known that any two such spaces are homeomorphic. 
The homeomorphism group of a topological space $X$ is written $\Homeo(X)$. 
The commutator subgroup of a group $\Gamma$ is denoted by $D(\Gamma)$. 
We let $\Gamma_\ab$ denote the abelianization $\Gamma/D(\Gamma)$. 

In this article, by an \'etale groupoid 
we mean a second countable locally compact Hausdorff groupoid 
such that the range map is a local homeomorphism. 
We refer the reader to \cite{Rtext,R08Irish} 
for background material on \'etale groupoids. 
Roughly speaking, a groupoid $\G$ is a `group-like' object, 
in which the product may not be defined for all pairs in $\G$. 
An \'etale groupoid $\G$ is equipped with locally compact Hausdorff topology, 
which is compatible with the groupoid structure, 
and the map $g\mapsto gg^{-1}$ is a local homeomorphism. 
For an \'etale groupoid $\G$, 
we let $\G^{(0)}$ denote the unit space and 
let $s$ and $r$ denote the source and range maps, 
i.e. $s(g)=g^{-1}g$ and $r(g)=gg^{-1}$. 
An element $g\in\G$ can be thought of as an arrow from $s(g)$ to $r(g)$. 
For $x\in\G^{(0)}$, 
$\G(x)=r(\G x)$ is called the $\G$-orbit of $x$. 
When every $\G$-orbit is dense in $\G^{(0)}$, 
$\G$ is said to be minimal. 
For a subset $Y\subset\G^{(0)}$, 
the reduction of $\G$ to $Y$ is $r^{-1}(Y)\cap s^{-1}(Y)$ and 
denoted by $\G|Y$. 
If $Y$ is clopen, then 
the reduction $\G|Y$ is an \'etale subgroupoid of $\G$ in an obvious way. 
For $x\in\G^{(0)}$, 
we write $\G_x=r^{-1}(x)\cap s^{-1}(x)$ and call it the isotropy group of $x$. 
The isotropy bundle of $\G$ is 
$\G'=\{g\in\G\mid r(g)=s(g)\}=\bigcup_{x\in\G^{(0)}}\G_x$. 
We say that $\G$ is principal if $\G'=\G^{(0)}$. 
When the interior of $\G'$ is $\G^{(0)}$, 
we say that $\G$ is essentially principal. 

A subset $U\subset\G$ is called a $\G$-set if $r|U,s|U$ are injective. 
Any open $\G$-set $U$ induces the homeomorphism 
$(r|U)\circ(s|U)^{-1}$ from $s(U)$ to $r(U)$. 
We write $\theta(U)=(r|U)\circ(s|U)^{-1}$. 
When $U,V$ are $\G$-sets, 
\[
U^{-1}=\{g\in\G\mid g^{-1}\in U\}
\]
and 
\[
UV=\{gg'\in\G\mid g\in U,\ g'\in V,\ s(g)=r(g')\}
\]
are also $\G$-sets. 
A probability measure $\mu$ on $\G^{(0)}$ is said to be $\G$-invariant 
if $\mu(r(U))=\mu(s(U))$ holds for every open $\G$-set $U$. 
The set of all $\G$-invariant probability measures is denoted by $M(\G)$. 

For an \'etale groupoid $\G$, 
we denote the reduced groupoid $C^*$-algebra of $\G$ by $C^*_r(\G)$ and 
identify $C_0(\G^{(0)})$ with a subalgebra of $C^*_r(\G)$. 
J. Renault obtained the following theorem 
(see also \cite[Theorem 5.1]{M12PLMS}). 

\begin{theorem}[{\cite[Theorem 5.9]{R08Irish}}]\label{Renault}
Two essentially principal \'etale groupoids $\G_1$ and $\G_2$ are isomorphic 
if and only if there exists an isomorphism $\phi:C^*_r(\G_1)\to C^*_r(\G_2)$ 
such that $\phi(C_0(\G_1^{(0)}))=C_0(\G_2^{(0)})$. 
\end{theorem}

\subsection{Examples}

In this subsection, 
we present several examples of \'etale groupoids. 
Throughout this subsection, by an \'etale groupoid, 
we mean a second countable \'etale groupoid 
whose unit space is the Cantor set. 

\begin{example}[AF groupoids]\label{AF/def}
We would like to recall the notion of AF groupoids 
(\cite[Definition III.1.1]{Rtext},\cite[Definition 3.7]{GPS04ETDS},
\cite[Definition 2.2]{M12PLMS}). 
Let $\G$ be an \'etale groupoid. 
\begin{itemize}
\item We say that $\K\subset\G$ is an elementary subgroupoid 
if $\K$ is a compact open principal subgroupoid of $\G$ 
such that $\K^{(0)}=\G^{(0)}$. 
\item We say that $\G$ is an AF groupoid 
if it can be written as an increasing union of elementary subgroupoids. 
\end{itemize}
If $\K$ is a compact \'etale principal groupoid, 
then $\K$ is identified with the equivalence relation 
$\{(r(g),s(g))\mid g\in\K\}$ on $\K^{(0)}$ and 
the topology on $\K$ agrees with 
the relative topology from $\K^{(0)}\times\K^{(0)}$. 
Also, the equivalence relation $\K$ is uniformly finite, i.e. 
there exists $n\in\N$ such that $\#r^{-1}(x)\leq n$ for any $x\in\K^{(0)}$. 

An AF groupoid is principal by definition. 
The $C^*$-algebra associated with an AF groupoid is an AF algebra. 
It is known that any AF groupoids are represented 
by Bratteli diagrams (see \cite[Theorem 3.9]{GPS04ETDS}). 
We provide a brief explanation of it. 
A directed graph $B=(V,E)$ is called a Bratteli diagram 
when $V=\bigcup_{n=0}^\infty V_n$ and $E=\bigcup_{n=1}^\infty E_n$ are 
disjoint unions of finite sets of vertices and edges 
with maps $i:E_n\to V_{n-1}$ and $t:E_n\to V_n$ 
both of which are surjective. 
Let 
\[
X_B=\left\{(e_n)_n\in\prod_nE_n
\mid e_n\in E_n,\ t(e_n)=i(e_{n+1})\quad\forall n\in\N\right\}. 
\]
The set $X_B$ endowed with the relative topology is called 
the infinite path space of $B$. 
Define an equivalence relation (i.e. principal groupoid) $\K_m$ by 
\[
\K_m=\{((e_n)_n,(f_n)_n)\in X_B\times X_B\mid
e_n=f_n\quad\forall n\geq m\}. 
\]
Then, $\K_m$ equipped with the relative topology from $X_B\times X_B$ 
is a compact principal \'etale groupoid. 
Clearly one has $\K_m\subset\K_{m+1}$. 
Set $\G=\bigcup_m\K_m$. 
Endowed with the inductive limit topology, $\G$ becomes an AF groupoid. 
Conversely, Theorem 3.9 of \cite{GPS04ETDS} states that 
any AF groupoid arises in such a way. 
\end{example}

\begin{example}[Transformation groupoids]\label{transf/def}
Let $\phi:\Gamma\curvearrowright X$ be 
an action of a countable discrete group $\Gamma$ 
on a Cantor set $X$ by homeomorphisms. 
We let $G_\phi=\Gamma\times X$ and define the following groupoid structure: 
$(\gamma,x)$ and $(\gamma',x')$ are composable 
if and only if $x=\phi^{\gamma'}(x')$, in which case 
$(\gamma,\phi^{\gamma'}(x'))\cdot(\gamma',x')=(\gamma\gamma',x')$, 
and $(\gamma,x)^{-1}=(\gamma^{-1},\phi^\gamma(x))$. 
Then $G_\phi$ is an \'etale groupoid and 
called the transformation groupoid 
arising from $\phi:\Gamma\curvearrowright X$. 
The unit space $\G_\phi^{(0)}$ is canonically identified with $X$ 
via the map $(1,x)\mapsto x$. 

The groupoid $\G_\phi$ is principal if and only if 
the action $\phi$ is free, that is, 
$\phi^\gamma$ does not have any fixed points unless $\gamma=1$. 
The groupoid $\G_\phi$ is essentially principal if and only if 
the action $\phi$ is topologically free, that is, 
$\{x\in X\mid\phi^\gamma(x)=x\}$ has no interior points unless $\gamma=1$. 
The groupoid $\G_\phi$ is minimal if and only if 
the action $\phi$ is minimal, that is, 
any orbit of $\phi$ is dense in $X$. 

The $C^*$-algebra $C^*_r(\G_\phi)$ is canonically isomorphic to 
the crossed product $C^*$-algebra $C(X)\rtimes_r\Gamma$. 
\end{example}

K. Medynets, R. Sauer and A. Thom recently obtained 
the following interesting result. 

\begin{theorem}[{\cite[Theorem 3.2]{MST15}}]
Let $\Gamma$ and $\Lambda$ be finitely generated groups. 
The following are equivalent. 
\begin{enumerate}
\item There exist free actions $\phi:\Gamma\curvearrowright X$ and 
$\psi:\Lambda\curvearrowright Y$ on Cantor sets 
such that $\G_\phi\cong\G_\psi$. 
\item $\Gamma$ and $\Lambda$ are bi-Lipschitz equivalent. 
\end{enumerate}
\end{theorem}

\begin{example}[SFT groupoids]\label{SFT/def}
We recall the definition of 
\'etale groupoids arising from one-sided shifts of finite type 
(\cite[Section 6.1]{M15crelle}). 
Let $(V,E)$ be a finite directed graph, 
where $V$ is a finite set of vertices 
and $E$ is a finite set of edges. 
For $e\in E$, $i(e)$ denotes the initial vertex of $e$ and 
$t(e)$ denotes the terminal vertex of $e$. 
Let $A=(A(\xi,\eta))_{\xi,\eta\in V}$ be 
the adjacency matrix of $(V,E)$, that is, 
\[
A(\xi,\eta)=\#\{e\in E\mid i(e)=\xi,\ t(e)=\eta\}. 
\]
We assume that $A$ is irreducible 
(i.e. for all $\xi,\eta\in V$ 
there exists $n\in\N$ such that $A^n(\xi,\eta)>0$) 
and that $A$ is not a permutation matrix. 
Define 
\[
X_A=\{(x_k)_{k\in\N}\in E^\N
\mid t(x_k)=i(x_{k+1})\quad\forall k\in\N\}. 
\]
With the product topology, $X_A$ is a Cantor set. 
Define a surjective continuous map $\sigma_A:X_A\to X_A$ by 
\[
\sigma_A(x)_k=x_{k+1}\quad k\in\N,\ x=(x_k)_k\in X_A. 
\]
In other words, $\sigma_A$ is the (one-sided) shift on $X_A$.  
It is easy to see that $\sigma_A$ is a local homeomorphism. 
The dynamical system $(X_A,\sigma_A)$ is called 
the one-sided irreducible shift of finite type (SFT) 
associated with the graph $(V,E)$ (or the matrix $A$). 

The \'etale groupoid $\G_A$ for $(X_A,\sigma_A)$ is given by 
\[
\G_A=\{(x,n,y)\in X_A\times\Z\times X_A\mid
\exists k,l\in\N,\ n=k{-}l,\ \sigma_A^k(x)=\sigma_A^l(y)\}. 
\]
The topology of $\G_A$ is generated by the sets 
$\{(x,k{-}l,y)\in\G_A\mid x\in P,\ y\in Q,\ \sigma_A^k(x)=\sigma_A^l(y)\}$, 
where $P,Q\subset X_A$ are open and $k,l\in\N$. 
Two elements $(x,n,y)$ and $(x',n',y')$ in $\G$ are composable 
if and only if $y=x'$, and the multiplication and the inverse are 
\[
(x,n,y)\cdot(y,n',y')=(x,n{+}n',y'),\quad (x,n,y)^{-1}=(y,-n,x). 
\]
We identify $X_A$ with the unit space $\G_A^{(0)}$ via $x\mapsto(x,0,x)$. 
We call $\G_A$ the SFT groupoid associated with the matrix $A$. 

The groupoid $\G_A$ is essentially principal and minimal. 

The groupoid $C^*$-algebra $C^*_r(\G_A)$ is isomorphic to 
the Cuntz-Krieger algebra $\mathcal{O}_A$ of \cite{CK80Invent}, 
which is simple and purely infinite. 
\end{example}

\section{Homology groups}

The homology groups $H_n(\G)$ of an \'etale groupoid $\G$ were first 
introduced and studied by M. Crainic and I. Moerdijk in \cite{CM00crelle}. 
In the case that the unit space $\G^{(0)}$ is a Cantor set, 
we investigated connections 
between the homology groups and dynamical properties of $\G$ 
in \cite{M12PLMS,M15crelle,M15}. 
In this section, 
we would like to recall the definition of $H_n(\G)$ 
for an \'etale groupoid $\G$ whose unit space is a Cantor set. 

Let $A$ be a topological abelian group. 
For a locally compact Hausdorff space $X$, we denote by $C_c(X,A)$ 
the set of $A$-valued continuous functions with compact support. 
When $X$ is compact, we simply write $C(X,A)$. 
With pointwise addition, $C_c(X,A)$ is an abelian group. 
Let $\pi:X\to Y$ be a local homeomorphism 
between locally compact Hausdorff spaces. 
For $f\in C_c(X,A)$, we define a map $\pi_*(f):Y\to A$ by 
\[
\pi_*(f)(y)=\sum_{\pi(x)=y}f(x). 
\]
It is not so hard to see that $\pi_*(f)$ belongs to $C_c(Y,A)$ and 
that $\pi_*$ is a homomorphism from $C_c(X,A)$ to $C_c(Y,A)$. 
Besides, if $\pi':Y\to Z$ is another local homeomorphism to 
a locally compact Hausdorff space $Z$, then 
one can check $(\pi'\circ\pi)_*=\pi'_*\circ\pi_*$ in a direct way. 
Thus, $C_c(\cdot,A)$ is a covariant functor 
from the category of locally compact Hausdorff spaces 
with local homeomorphisms 
to the category of abelian groups with homomorphisms. 

Let $\G$ be an \'etale groupoid. 
For $n\in\N$, we write $\G^{(n)}$ 
for the space of composable strings of $n$ elements in $\G$, that is, 
\[
\G^{(n)}=\{(g_1,g_2,\dots,g_n)\in\G^n\mid
s(g_i)=r(g_{i+1})\text{ for all }i=1,2,\dots,n{-}1\}. 
\]
For $i=0,1,\dots,n$, 
we let $d_i:G^{(n)}\to G^{(n-1)}$ be a map defined by 
\[
d_i(g_1,g_2,\dots,g_n)=\begin{cases}
(g_2,g_3,\dots,g_n) & i=0 \\
(g_1,\dots,g_ig_{i+1},\dots,g_n) & 1\leq i\leq n{-}1 \\
(g_1,g_2,\dots,g_{n-1}) & i=n. 
\end{cases}
\]
When $n=1$, we let $d_0,d_1:\G^{(1)}\to\G^{(0)}$ be 
the source map and the range map, respectively. 
Clearly the maps $d_i$ are local homeomorphisms. 

Define the homomorphisms $\partial_n:C_c(\G^{(n)},A)\to C_c(\G^{(n-1)},A)$ 
by 
\[
\partial_n=\sum_{i=0}^n(-1)^id_{i*}. 
\]
It is easy see that 
the abelian groups $C_c(\G^{(n)},A)$ 
together with the boundary operators $\partial_n$ form a chain complex. 

\begin{definition}
[{\cite[Section 3.1]{CM00crelle},\cite[Definition 3.1]{M12PLMS}}]
\label{homology}
We let $H_n(\G,A)$ be the homology groups of the Moore complex above, 
i.e. $H_n(\G,A)=\ker\partial_n/\Ima\partial_{n+1}$, 
and call them the homology groups of $\G$ with constant coefficients $A$. 
When $A=\Z$, we simply write $H_n(\G)=H_n(\G,\Z)$. 
In addition, we define 
\[
H_0(\G)^+=\{[f]\in H_0(\G)\mid f(x)\geq0\text{ for all }x\in\G^{(0)}\}, 
\]
where $[f]$ denotes the equivalence class of $f\in C_c(\G^{(0)},\Z)$. 
\end{definition}

\begin{remark}
The pair $(H_0(\G),H_0(\G)^+)$ is not necessarily an ordered abelian group 
in general, 
because $H_0(\G)^+\cap({-}H_0(\G)^+)$ may not equal $\{0\}$. 
In fact, 
when $\G$ is the SFT groupoid, $H_0(G)^+=H_0(G)$. 
\end{remark}

\begin{example}\label{ExofH}
\begin{enumerate}
\item Let $\G$ be an AF groupoid (see Example \ref{AF/def}). 
There exists an isomorphism $\pi:H_0(\G)\to K_0(C^*_r(\G))$ such that 
$\pi(H_0(\G)^+)=K_0(C^*_r(\G))^+$ and $\pi([1_{\G^{(0)}}])=[1]$ 
(\cite[Theorem 4.10]{M12PLMS}). 
For $n\geq1$, we have $H_n(\G)=0$ (\cite[Theorem 4.11]{M12PLMS}). 
\item Let $\G_\phi$ be the transformation groupoid 
associated with a group action $\phi:\Gamma\curvearrowright X$ 
(see Example \ref{transf/def}). 
Then the homology groups $H_n(\G)$ is naturally isomorphic to 
the usual group homology $H_n(\Gamma,C(X,\Z))$ 
of $\Gamma$ with coefficients in $C_c(X,\Z)$ (see \cite[Chapter III]{BrGTM}). 
\item Let $\G_A$ be an SFT groupoid, 
where $A$ is the adjacency matrix of 
an irreducible finite directed graph $(V,E)$ (see Example \ref{SFT/def}). 
The matrix $A$ acts on the abelian group $\Z^V$ by multiplication. 
Then one has 
\[
H_n(\G_A)\cong\begin{cases}\Coker(\id-A^t)&n=0\\
\Ker(\id-A^t)&n=1\\0&n\geq2. \end{cases}
\]
See \cite[Theorem 4.14]{M12PLMS}. 
\end{enumerate}
\end{example}

The cohomology groups $H^n(\G)$ of an \'etale groupoid $\G$ were 
introduced by J. Renault in \cite{Rtext} and 
have been studied by many authors. 
When $\G_\phi$ is the transformation groupoid 
associated with a group action $\phi:\Gamma\curvearrowright X$, 
the cohomology $H^n(\G_\phi)$ is canonically isomorphic to 
the usual group cohomology $H^n(\Gamma,C(X,\Z))$. 
In particular, when $\Gamma=\Z^N$, 
there exist natural isomorphisms $H_n(\G_\phi)\cong H^{N-n}(\G_\phi)$ 
(Poincar\'e duality). 
In general, however, 
we do not know if any connections exist between $H_*(\G)$ and $H^*(\G)$. 

For the homology groups $H_n(\G)$, the following K\"unneth theorem holds. 

\begin{theorem}[{\cite[Theorem 2.4]{M15}}]\label{Kunneth}
Let $\G$ and $\H$ be \'etale groupoids. 
For any $n\geq0$, there exists a natural short exact sequence 
\[
0\longrightarrow\bigoplus_{i+j=n}H_i(\G)\otimes H_j(\H)
\longrightarrow H_n(\G\times\H)
\longrightarrow\bigoplus_{i+j=n-1}\Tor(H_i(\G),H_j(\H))\longrightarrow0. 
\]
Furthermore these sequences split (but not canonically). 
\end{theorem}

In \cite[Section 2.3]{M15}, 
we made the following conjecture 
about homology groups $H_n(\G)$ and $K$-groups $K_i(C^*_r(\G))$. 

\begin{conjecture}[HK conjecture]\label{HK}
Let $\G$ be an essentially principal minimal \'etale groupoid 
whose unit space $\G^{(0)}$ is a Cantor set. 
Then we have 
\[
\bigoplus_{i=0}^\infty H_{2i}(\G)\cong K_0(C^*_r(\G))
\]
and 
\[
\bigoplus_{i=0}^\infty H_{2i+1}(\G)\cong K_1(C^*_r(\G)). 
\]
\end{conjecture}

\begin{example}
The HK conjecture is true for any AF groupoid $\G$. 
This is clear from Example \ref{ExofH} (1). 
\end{example}

\begin{example}
Let $\G_\phi$ be the transformation groupoid 
associated with a group action $\phi:\Gamma\curvearrowright X$. 
In general, 
it is not known whether the HK conjecture holds for $\G_\phi$. 
When $\Gamma=\Z$, 
the Pimsner-Voiculescu exact sequence implies that 
the HK conjecture is true. 
Suppose $\Gamma=\Z^N$. 
The suspension space $\Omega_\phi$ is 
the quotient space of $\R^N\times X$ by the equivalence relation 
\[
(t,x)\sim(t',x')\iff t-t'\in\Z^N,\ \phi^{t-t'}(x)=x'. 
\]
The translation $(t,x)\mapsto (t+s,x)$ gives rise to 
an action $\tilde\phi:\R^N\curvearrowright\Omega_\phi$. 
It is well-known that $C(\Omega_\phi)\rtimes_{\tilde\phi}\R^N$ is 
stably isomorphic to $C(X)\rtimes_\phi\Z^N$. 
Then we have 
\begin{align*}
K_i(C^*_r(\G_\phi))&=K_i(C(X)\rtimes_\phi\Z^N)
\cong K_i(C(\Omega_\phi)\rtimes_{\tilde\phi}\R^N)\\
&\cong K_{i+N}(C(\Omega_\phi))\qquad\because\text{Thom Isomorphism}\\
&=K^{i+N}(\Omega_\phi). 
\end{align*}
On the other hand, 
we know $H_n(\G_\phi)=H_n(\Z^N,C(X,\Z))$ is naturally isomorphic to 
$\check H^n(\Omega_\phi,\Z)$, 
where $\check H$ denotes the \v{C}ech cohomology. 
(This is a folklore fact and I don't know an appropriate reference. 
A relevant remark can be found 
in the final paragraph of \cite[Chapter III.1]{BrGTM}.) 
By the Chern character, there exist isomorphisms 
$K^*(\Omega_\phi)\otimes\Q\to\bigoplus H^*(\Omega_\phi,\Z)\otimes\Q$. 
It follows that there exist isomorphisms 
\[
\bigoplus_{i=0}^\infty H_{2i}(\G_\phi)\otimes\Q
\cong K_0(C^*_r(\G_\phi))\otimes\Q
\]
and 
\[
\bigoplus_{i=0}^\infty H_{2i+1}(\G_\phi)\otimes\Q
\cong K_1(C^*_r(\G_\phi))\otimes\Q. 
\]
The HK conjecture asks 
if the integral version of these isomorphisms is true or not. 

For more general group actions $\Gamma\curvearrowright X$, 
the HK conjecture is wide open. 
\end{example}

\begin{example}
Let $\G_A$ be an SFT groupoid, 
where $A$ is the adjacency matrix of 
an irreducible finite directed graph $(V,E)$. 
By Example \ref{ExofH} (3), 
we can see $H_i(\G_A)\cong K_i(C^*_r(\G_A))$ for $i=1,2$. 
Therefore, the HK conjecture holds for $\G_A$. 
\end{example}

\section{Topological full groups}

In this section, 
we introduce the definition of topological full groups. 

\begin{definition}[{\cite[Definition 2.3]{M12PLMS}}]\label{tfg/def}
Let $\G$ be an essentially principal \'etale groupoid 
whose unit space $\G^{(0)}$ is a Cantor set. 
The set of all $\alpha\in\Homeo(\G^{(0)})$ for which 
there exists a compact open $\G$-set $U$ 
satisfying $\alpha=\theta(U)$ 
is called the topological full group of $\G$ and denoted by $[[\G]]$. 
\end{definition}

For $\alpha\in[[\G]]$ the compact open $\G$-set $U$ as above uniquely exists, 
because $\G$ is essentially principal. 
Obviously $[[\G]]$ is a subgroup of $\Homeo(\G^{(0)})$. 
Since $\G$ is second countable, it has countably many compact open subsets, 
and so $[[\G]]$ is at most countable. 

A homeomorphism $\alpha:\G^{(0)}\to\G^{(0)}$ belongs to $[[\G]]$ 
if and only if for any $x\in\G^{(0)}$ 
there exists a compact open $\G$-set $V$ 
such that $x$ is in $s(V)$ and 
$\alpha$ equals $\theta(V)$ on a neighborhood of $x$. 
Thus, $\alpha$ is in $[[\G]]$ if and only if 
the `graph' of $\alpha$ is a clopen subset of $\G$. 

\begin{example}\label{AF/tfg}
Let $\G$ be an AF groupoid 
arising from a Bratteli diagram $(V,E)$ (see Example \ref{AF/def}). 
For each $k\in\N$ and $v\in V_k$, 
we let $E_v$ be the set of paths from a vertex in $V_0$ to the vertex $v$, 
i.e. 
\[
E_v=\{(e_1,e_2,\dots,e_k)\mid t(e_n)=i(e_{n+1}),\ t(e_k)=v\}. 
\]
Suppose that a permutation $\sigma_v$ on $E_v$ is given. 
Then we can define $\tilde\sigma_v\in[[\G]]$ by 
\[
\tilde\sigma_v((e_n)_n)
=\begin{cases}(\sigma_v(e_1,e_2,\dots,e_k),e_{k+1},\dots)&t(e_k)=v\\
(e_n)_n&t(e_k)\neq v. \end{cases}
\]
Let $G_k\subset[[\G]]$ be a subgroup generated by 
$\tilde\sigma_v$ for vertices $v\in V_k$ and permutations $\sigma_v$. 
It is easy to see 
\[
G_k\cong\bigoplus_{v\in V_k}\mathfrak{S}_{\#E_v},\quad 
G_k\subset G_{k+1}
\]
and $[[\G]]=\bigcup_kG_k$. 
Thus, $[[\G]]$ is an increasing union of subgroups isomorphic to 
finite direct sums of symmetric groups. 
Conversely, it is known that if $[[\G]]$ is locally finite, then 
$\G$ is AF (\cite[Proposition 3.2]{M06IJM}). 
\end{example}

\begin{example}\label{transf/tfg}
Let $\G_\phi$ be the transformation groupoid 
associated with a group action $\phi:\Gamma\curvearrowright X$ 
(see Example \ref{transf/def}). 
Suppose that $\G_\phi$ is essentially principal. 
Take $\alpha\in[[\G_\phi]]$. 
There exists a compact open $\G_\phi$-set $U$ such that $\alpha=\theta(U)$. 
We can find a continuous map $c:X\to\Gamma$ 
such that $U=\{(c(x),x)\mid x\in X\}$. 
It follows that $\alpha(x)=\phi^{c(x)}(x)$ for all $x\in X$. 
\end{example}

\begin{example}\label{SFT/tfg}
Let $\G_A$ be an SFT groupoid, 
where $A$ is the adjacency matrix of 
an irreducible finite directed graph $(V,E)$ (see Example \ref{SFT/def}). 
We say that $\mu=(e_1,e_2,\dots,e_k)\in E^k$ is a path 
if $t(e_j)=i(e_{j+1})$ for every $j=1,2,\dots,k{-}1$. 
The terminal vertex of $\mu$ is written $t(\mu)=t(e_k)$. 
The length of $\mu$ is written $\lvert\mu\rvert=k$. 
For a path $\mu=(e_1,e_2,\dots,e_k)$, 
\[
C_\mu=\{(x_n)_{n\in\N}\in X_A\mid x_j=e_j\quad\forall j=1,2,\dots,k\}
\]
is a clopen subset of $X_A$ and is called a cylinder set. 
For two paths $\mu$ and $\nu$ with $t(\mu)=t(\nu)$, 
we define a compact open $\G_A$-set $U_{\mu,\nu}$ by 
\[
U_{\mu,\nu}=\{(x,\lvert\mu\rvert{-}\lvert\nu\rvert,y)\in\G_A\mid
\sigma_A^{\lvert\mu\rvert}(x)=\sigma^{\lvert\nu\rvert}(y),\ 
x\in C_\mu,\ y\in C_\nu\}. 
\]
The subsets $U_{\mu,\nu}$ form a base for the topology of $\G_A$. 

Take $\alpha\in[[\G_A]]$. 
There exist $\mu_1,\mu_2,\dots,\mu_m$ and $\nu_1,\nu_2,\dots,\nu_m$ 
such that the following hold. 
\begin{itemize}
\item $\{C_{\mu_i}\mid i=1,2,\dots,m\}$ is a clopen partition of $X_A$. 
\item $\{C_{\nu_i}\mid i=1,2,\dots,m\}$ is a clopen partition of $X_A$. 
\item $t(\mu_i)=t(\nu_i)$ for every $i$. 
\item $U=\bigcup_iU_{\mu_i,\nu_i}$ is a compact open $\G$-set 
satisfying $\alpha=\theta(U)$. 
\end{itemize}

Let us consider the simplest case, 
namely that $A$ is a $1\times1$ matrix $[n]$. 
The groupoid $C^*$-algebra $C^*_r(\G_{[n]})$ is 
the Cuntz algebra $\mathcal{O}_n$. 
V. V. Nekrashevych \cite[Proposition 9.6]{Ne04JOP} observed that 
the topological full group $[[\G_{[n]}]]$ is naturally isomorphic to 
the Higman-Thompson group $V_{n,1}$.  
The group $V_{n,r}$ is defined to be the group of 
all right continuous PL bijections $f:[0,r)\to[0,r)$ 
with finitely many singularities 
such that all singularities of $f$ are in $\Z[1/n]$, 
the derivative of $f$ at any non-singular point is $n^k$ for some $k\in\Z$ 
and $f$ maps $\Z[1/n]\cap[0,r)$ to itself. 
The isomorphism $\pi:[[\G_{[n]}]]\to V_{n,1}$ is described as follows. 
We identify the shift space $X_{[n]}$ with $\{0,1,\dots,n{-}1\}^\N$. 
Define a continuous map $F:X_{[n]}\to[0,1]$ by 
\[
F((x_k)_k)=\sum_{k=1}^\infty\frac{x_k}{n^k}. 
\]
Then we have $\pi(\alpha)\circ F=F\circ\alpha$ 
for any $\alpha\in[[\G_{[n]}]]$. 

For an SFT groupoid $\G_A$, 
$[[\G_A]]$ is thought of as a generalization of 
the Higman-Thompson group $V_{n,1}$. 
\end{example}

When $\alpha=\theta(U)$ is an element of $[[\G]]$, 
$1_U\in C_c(\G)$ can be thought of as a unitary in $C^*_r(\G)$. 
This unitary $1_U$ normalizes $C(\G^{(0)})$, 
namely $1_Uf1_U^*=f\circ\alpha$ holds for every $f\in C(\G^{(0)})$. 
Conversely, if $u\in C^*_r(\G)$ is a normalizer of $C(\G^{(0)})$ 
and $ufu^*=f\circ\alpha$, then $\alpha$ belongs to $[[\G]]$. 
Thus, we have the following. 

\begin{proposition}[{\cite[Proposition 5.6]{M12PLMS}}]
Suppose that $\G$ is an essentially principal \'etale groupoid 
whose unit space $\G^{(0)}$ is a Cantor set. 
There exists a natural short exact sequence 
\[
\begin{CD}
1@>>>U(C(\G^{(0)}))@>>>
N(C(\G^{(0)}),C^*_r(\G))@>\sigma>>[[\G]]@>>>1, 
\end{CD}
\]
where $N(C(\G^{(0)}),C^*_r(\G))$ denotes the group of 
unitary normalizers of $C(\G^{(0)})$ in $C^*_r(\G)$. 
Furthermore, the homomorphism $\sigma$ has a right inverse. 
\end{proposition}

Next, we would like to introduce the index map $I:[[\G]]\to H_1(\G)$. 

\begin{definition}[{\cite[Definition 7.1]{M12PLMS}}]
Let $\G$ be an essentially principal \'etale groupoid 
whose unit space is a Cantor set. 
For $\alpha\in[[\G]]$, 
a compact open $\G$-set $U$ satisfying $\alpha=\theta(U)$ uniquely exists. 
It is easy to see that $1_U$ is a $1$-cycle, 
i.e. $1_U\in\Ker\partial_1$ (see Section 3). 
We define a map $I:[[\G]]\to H_1(G)$ by $I(\alpha)=[1_U]$ 
and call it the index map. 
\end{definition}

It is easy to check that $I:[[\G]]\to H_1(\G)$ is a homomorphism. 
We write $[[\G]]_0=\Ker I$. 
For $\alpha=\theta(U)\in[[\G]]$, 
the element $1_U$ may be regarded as a unitary of $C^*_r(\G)$, 
and so we can think about its $K_1$-class $[1_U]\in K_1(C^*_r(\G))$. 
It is a natural open question to find a connection 
between $[1_U]\in K_1(C^*_r(\G))$ and $I(\alpha)\in H_1(\G)$. 

We made the following conjecture 
about abelianization $[[\G]]_\ab$ and homology groups $H_n(\G)$ 
in \cite[Section 2.3]{M15}. 

\begin{conjecture}[AH conjecture]\label{AH}
Let $\G$ be an essentially principal minimal \'etale groupoid 
whose unit space $\G^{(0)}$ is a Cantor set. 
Then there exists an exact sequence 
\[
\begin{CD}
H_0(\G)\otimes\Z_2@>>>[[\G]]_\ab@>I>>H_1(\G)@>>>0. 
\end{CD}
\]
Especially, if $H_0(\G)$ is $2$-divisible, 
then we have $[[\G]]_\ab\cong H_1(\G)$. 
\end{conjecture}

Indeed, in many examples we can verify that 
there exists a short exact sequence 
\[
\begin{CD}
0@>>>H_0(\G)\otimes\Z_2@>j>>[[\G]]_\ab@>I>>H_1(\G)@>>>0. 
\end{CD}
\]
In such a case, 
we say that $\G$ satisfies the strong AH property. 
There exists $\G$ which does not satisfy the strong AH property 
(but the AH conjecture is still true). 
See \cite[Example 7.1]{Ne15}, \cite[Section 5.5]{M15} 
or Theorem \ref{prodSFT3} (4). 

The homomorphism $H_0(\G)\otimes\Z_2\to[[\G]]_\ab$ 
appearing in the AH conjecture is described as follows. 
Let $U$ be a compact open $\G$-set 
satisfying $r(U)\cap s(U)=\emptyset$. 
Define $\tau\in[[\G]]$ by 
\[
\tau(x)=\begin{cases}\theta(U)(x)&x\in s(U)\\
\theta(U^{-1})(x)&x\in r(U)\\x&\text{otherwise. }\end{cases}
\]
The homomorphism $H_0(\G)\otimes\Z_2\to[[\G]]_\ab$ sends 
the equivalence class of $1_{s(U)}\otimes\bar{1}$ to 
the equivalence class of $\tau$. 
Notice that it is not clear at all if this is really well-defined. 
One can find a proof of the well-definedness in \cite{Ne15}. 

\begin{example}\label{ExofAH}
\begin{enumerate}
\item Let $\G$ be an AF groupoid. 
In Example \ref{AF/tfg}, we have seen that 
$[[\G]]$ is an increasing union of 
the subgroups $G_k\cong\bigoplus_{v\in V_k}\mathfrak{S}_{\#E_v}$. 
Clearly $(G_k)_\ab$ is isomorphic to $\bigoplus_{v\in V_k}\Z_2$. 
So, $[[\G]]_\ab$ is an inductive limit of $\bigoplus_{v\in V_k}\Z_2$, 
and the connecting maps are given by the edge sets $E_k$. 
Hence $[[\G]]_\ab$ is isomorphic to $H_0(\G)\otimes\Z_2$ 
(see \cite[Section 3]{M06IJM}). 
In particular, the AF groupoid $\G$ has the strong AH property. 
\item Let $\G_\phi$ be the transformation groupoid 
associated with a minimal group action $\phi:\Gamma\curvearrowright X$. 
When $\Gamma$ is $\Z$, 
it was shown that $[[\G_\phi]]_\ab$ is isomorphic to 
$H_1(\G_\phi)\oplus(H_0(\G_\phi)\otimes\Z_2)$ (\cite[Section 4]{M06IJM}). 
Thus, $\G_\phi$ has the strong AH property. 
When $\Gamma$ is $\Z^N$, 
we can prove that the AH conjecture holds for $\G_\phi$ 
(see Theorem \ref{almstfnt/thm} (3)). 
But, we do not know 
whether or not $\G_\phi$ satisfies the strong AH property. 
For other group actions $\Gamma\curvearrowright X$, nothing is known. 
\item Let $\G_A$ be an SFT groupoid. 
It was proved that the abelianization $[[\G_A]]_\ab$ is isomorphic to 
$H_1(\G_A)\oplus(H_0(\G_A)\otimes\Z_2)$ (\cite[Corollary 6.24]{M15crelle}). 
Thus, $\G_A$ has the strong AH property. 
See Theorem \ref{SFT/thm}. 
\end{enumerate}
\end{example}

\section{Isomorphism theorem}

In this section, we discuss the following theorem. 

\begin{theorem}[{\cite[Theorem 3.10]{M15crelle}}]\label{isomorphism}
For $i=1,2$, let $\G_i$ be an essentially principal \'etale groupoid 
whose unit space is a Cantor set. 
Suppose that $\G_i$ is minimal. 
The following conditions are equivalent. 
\begin{enumerate}
\item $\G_1$ and $\G_2$ are isomorphic as \'etale groupoids. 
\item $[[\G_1]]$ and $[[\G_2]]$ are isomorphic as discrete groups. 
\item $[[\G_1]]_0$ and $[[\G_2]]_0$ are isomorphic as discrete groups. 
\item $D([[\G_1]])$ and $D([[\G_2]])$ are isomorphic as discrete groups. 
\end{enumerate}
\end{theorem}

It it clear that condition (1) implies the other conditions. 
The reverse implications are nontrivial, 
and the essential part of the proof is contained in the following proposition. 

\begin{proposition}[{\cite[Theorem 3.5, Proposition 3.6]{M15crelle}}]
For $i=1,2$, let $\G_i$ be an essentially principal \'etale groupoid 
whose unit space is a Cantor set. 
Suppose that $\G_i$ is minimal. 
For each $i=1,2$, let $\Gamma_i$ be a subgroup of $[[\G_i]]$ 
such that $D([[\G_i]])\subset\Gamma_i$. 
If there exists an isomorphism $\Phi:\Gamma_1\to\Gamma_2$, then 
there exists a homeomorphism $\phi:\G_1^{(0)}\to\G_2^{(0)}$ 
such that $\Phi(\alpha)=\phi\circ\alpha\circ\phi^{-1}$ 
for all $\alpha\in\Gamma_1$. 
\end{proposition}

This proposition says that 
any isomorphism between the groups $\Gamma_1$ and $\Gamma_2$ is 
spatially realized by a homeomorphism between the unit spaces. 
If we get a homeomorphism $\phi:\G_1^{(0)}\to\G_2^{(0)}$ 
such that $\Phi(\alpha)=\phi\circ\alpha\circ\phi^{-1}$, 
then it is not so hard to see that 
$\phi$ gives rise to an isomorphism from $\G_1$ to $\G_2$. 
Therefore we can prove 
the remaining implications of Theorem \ref{isomorphism}. 

In the setting of topological dynamical systems, 
Theorem \ref{isomorphism} was first proved 
by T. Giordano, I. F. Putnam and C. F. Skau \cite{GPS99Israel} 
for minimal $\Z$-actions. 
In order to obtain the spatial realization (see the proposition above), 
they imported the method of H. Dye, 
who proved the same isomorphism result 
for measure preserving ergodic actions on Lebesgue spaces. 
We remark that for minimal $\Z$-actions $\phi_1$ and $\phi_2$, 
$\G_{\phi_1}$ is isomorphic to $\G_{\phi_2}$ 
if and only if $\phi_1$ is flip conjugate to $\phi_2$ 
(\cite[Theorem 2.4]{GPS95crelle}). 
See Theorem \ref{Z/flip}. 
Later, 
similar results were obtained by 
S. Bezuglyi and K. Medynets \cite{BM08Coll} and by Medynets \cite{Me11BLMS}. 
The proof of Theorem \ref{isomorphism} given in \cite{M15crelle} is 
along the same line as these works. 

The proposition above can be also thought of 
as an immediate consequence of 
the following theorem of M. Rubin \cite{Ru89TAMS} 
(see also \cite[Section 9]{Brin04GeomD}, \cite[Section 3.3]{Ne15}). 
Let $X$ be a topological space. 
We say that a subgroup $\Gamma\subset\Homeo(X)$ is locally dense 
if for every $x\in X$ and every open set $U\subset X$ with $x\in U$, 
the closure of 
\[
\{f(x)\mid f\in\Gamma,\ f|(X\setminus U)=\id|(X\setminus U)\}
\]
has nonempty interior. 

\begin{theorem}[{\cite[Corollary 3.5]{Ru89TAMS}}] 
For $i=1,2$, let $X_i$ be a locally compact, Hausdorff topological spaces 
without isolated points and 
let $\Gamma_i\subset\Homeo(X_i)$ be subgroups. 
If $\Gamma_1$ and $\Gamma_2$ are isomorphic and are both locally dense, 
then for any isomorphism $\Phi:\Gamma_1\to\Gamma_2$ 
there exists a unique homeomorphism $\phi:X_1\to X_2$ 
such that $\Phi(\alpha)=\phi\circ\alpha\circ\phi^{-1}$ 
for all $\alpha\in\Gamma_1$. 
\end{theorem}

Recently, V. V. Nekrashevych \cite{Ne15} introduced two normal subgroups 
$\A(\G)\subset\S(\G)\subset[[\G]]$. 
Roughly speaking, 
$\S(\G)$ is the subgroup generated by all elements of order two, 
and $\A(\G)$ is the subgroup generated by all elements of order three 
(see \cite{Ne15} for the precise definitions). 
They are analogs of the symmetric and alternating groups. 
He proved that 
the same statement as Theorem \ref{isomorphism} is true 
for $\A(\G)$ and $\S(\G)$.

\section{Almost finite groupoids}

In this section, 
we list known and unknown properties of almost finite groupoids. 
Let us begin with the definition. 

\begin{definition}[{\cite[Definition 6.2]{M12PLMS}}]\label{almstfnt/def}
Let $\G$ be an essentially principal \'etale groupoid 
whose unit space is a Cantor set. 
We say that $\G$ is almost finite 
if for any compact subset $C\subset\G$ and $\ep>0$ 
there exists an elementary subgroupoid $\K\subset\G$ such that 
\[
\frac{\#(C\K x\setminus\K x)}{\#(\K x)}<\ep
\]
for all $x\in\G^{(0)}$. 
We also remark that $\#(\K(x))$ equals $\#(\K x)$, because $\K$ is principal. 
\end{definition}

I remark that 
the idea of the definition above has its origin 
in the work of F. Latr\'emoli\`ere and N. Ormes \cite{LaO12Rocky}. 
More precisely, 
the notion of almost finiteness was made 
so that the arguments of \cite{LaO12Rocky} proceed in an analogous way. 

AF groupoids (see Example \ref{AF/def}) are almost finite. 
Indeed, any compact subset $C\subset\G$ is contained 
in an elementary subgroupoid $\K\subset\G$. 
As the next proposition shows, 
there exist almost finite groupoids which are not AF. 

\begin{proposition}[{\cite[Lemma 6.3]{M12PLMS}}]\label{Z^N>almstfnt}
When $\phi:\Z^N\curvearrowright X$ is a free action of $\Z^N$ 
on a Cantor set $X$, 
the transformation groupoid $\G_\phi$ is almost finite. 
\end{proposition}

Definition \ref{almstfnt/def} may remind the reader 
of the F\o lner condition for amenable groups. 
It may be natural to expect that 
transformation groupoids arising from free actions of amenable groups 
are almost finite. 
This is an important open question. 
In fact, X. Li recently proved that the converse is true. 

\begin{proposition}[{\cite{Li}}]
Let $\phi:\Gamma\curvearrowright X$ be a topologically free action 
of a discrete countable group $\Gamma$ on a Cantor set $X$. 
If $\G_\phi$ is almost finite, then $\Gamma$ is amenable. 
\end{proposition}

It is also an interesting problem 
to compare the notion of almost finite groupoids with almost AF groupoids, 
which were introduced by N. C. Phillips \cite{Ph05CMP}. 

For a $\G$-invariant probability measure $\mu\in M(\G)$, 
we can define a homomorphism $\hat\mu:H_0(\G)\to\R$ by 
\[
\hat\mu([f])=\int f\,d\mu. 
\]
It is clear that 
$\hat\mu([1_{\G^{(0)}}])=1$ and $\hat\mu(H_0(\G)^+)\subset[0,\infty)$. 
Thus $\hat\mu$ is a state on $(H_0(\G),H_0(\G)^+,[1_{\G^{(0)}}])$. 
It is also easy to see that the map $\mu\mapsto\hat\mu$ gives 
an isomorphism from $M(\G)$ to the state space. 

\begin{theorem}[{\cite[Theorem 3.4]{M15}}]\label{almstfnt/H0}
Let $\G$ be a minimal almost finite groupoid. 
\begin{enumerate}
\item $(H_0(\G),H_0(\G)^+)$ is a simple, weakly unperforated, 
ordered abelian group with the Riesz interpolation property. 
\item The homomorphism $\rho:H_0(\G)\to\Aff(M(\G))$ 
defined by $\rho([f])(\mu)=\hat\mu([f])$ has uniformly dense range, 
where $\Aff(M(\G))$ denotes 
the space of $\R$-valued affine continuous functions on $M(\G)$. 
\end{enumerate}
\end{theorem}

For topological full groups of almost finite groupoids, 
the following are known. 

\begin{theorem}[{\cite{M12PLMS,M15crelle,M15}}]\label{almstfnt/thm}
Let $\G$ be an almost finite groupoid. 
\begin{enumerate}
\item If $\G$ is minimal, 
then the commutator subgroup $D([[\G]])$ is simple. 
\item The index map $I:[[G]]\to H_1(\G)$ is surjective. 
\item If $\G$ is principal and minimal, then 
the AH conjecture holds for $\G$. 
\end{enumerate}
\end{theorem}

V. V. Nekrashevych \cite{Ne15} recently proved that 
if an almost finite groupoid $\G$ is minimal and expansive, 
then $D([[\G]])$ is finitely generated 
(see \cite{Ne15} for the definition of expansive groupoids). 

In general, it is not known 
if every almost finite groupoid $\G$ satisfies the strong AH property or not. 
In other words, we cannot prove that 
the homomorphism $j:H_0(\G)\otimes\Z_2\to[[\G]]_\ab$ is always injective, 
and cannot find an example such that the map $j$ has nontrivial kernel. 
However, for a minimal free action $\phi:\Z^N\curvearrowright X$, 
one can prove that the kernel of $j$ is at least `contained' 
in the infinitesimal subgroup of $H_0(\G_\phi)$, which is defined by 
\[
\Inf(H_0(\G_\phi))
=\left\{[f]\in H_0(\G_\phi)\mid
\int f\,d\mu=0\quad\forall\mu\in M(\G_\phi)\right\}. 
\]

\begin{proposition}[{\cite[Proposition 3.7]{M15}}]
Let $\phi:\Z^N\curvearrowright X$ be a minimal free action of $\Z^N$ 
on a Cantor set $X$. 
The kernel of the homomorphism $j:H_0(\G_\phi)\otimes\Z_2\to[[\G_\phi]]_\ab$ 
is contained in $\Inf(H_0(\G_\phi))\otimes\Z_2$. 
In particular, when $\Inf(H_0(\G_\phi))\otimes\Z_2$ is trivial, 
$\G_\phi$ has the strong AH property. 
\end{proposition}

Let $\phi:\Gamma\curvearrowright X$ be a free minimal action 
of an amenable group $\Gamma$ on a Cantor set $X$. 
It is known that the topological full group $[[\G_\phi]]$ is sofic 
(\cite[Proposition 5.1 (1)]{ES05MathAnn}, 
see also \cite[Section 3]{EM13PAMS}). 
For $\Gamma=\Z$, 
we discuss the amenability of $[[\G_\phi]]$ in Section 8.1.

\section{Purely infinite groupoids}

In this section, 
we list known and unknown properties of purely infinite groupoids. 
Let us begin with the definition. 

\begin{definition}[{\cite[Definition 4.9]{M15crelle}}]\label{pi/def}
Let $\G$ be an essentially principal \'etale groupoid 
whose unit space is a Cantor set. 
\begin{enumerate}
\item A clopen set $A\subset\G^{(0)}$ is said to be properly infinite 
if there exist compact open $\G$-sets $U,V\subset\G$ 
such that $s(U)=s(V)=A$, $r(U)\cup r(V)\subset A$ 
and $r(U)\cap r(V)=\emptyset$. 
\item We say that $\G$ is purely infinite 
if every clopen set $A\subset\G^{(0)}$ is properly infinite. 
\end{enumerate}
\end{definition}

If $\G^{(0)}$ is properly infinite, then 
the space $M(\G)$ of $\G$-invariant probability measure on $\G^{(0)}$ 
is empty. 
In addition, $[[\G]]$ contains the free product $\Z_2*\Z_3$ 
(\cite[Proposition 4.10]{M15crelle}), and hence $[[\G]]$ is not amenable. 
If $\G$ is purely infinite, then 
for any nonempty clopen set $Y\subset\G^{(0)}$, 
the reduction $\G|Y$ is again purely infinite. 
The SFT groupoids $\G_A$ (see Example \ref{SFT/def}) are 
typical examples of purely infinite minimal groupoids. 
When $\G$ is purely infinite and minimal, 
the reduced groupoid $C^*$-algebra $C^*_r(\G)$ is purely infinite and simple. 

Compare the following theorem with Theorem \ref{almstfnt/H0}

\begin{theorem}[{\cite[Lemma 5.3]{M15crelle}}]
Let $\G$ be a purely infinite groupoid. 
For any $h\in H_0(\G)$, 
there exists a non-empty clopen set $A\subset\G^{(0)}$ 
such that $[1_A]=h$. 
In particular, $H_0(\G)=H_0(\G)^+$. 
\end{theorem}

For topological full groups of purely infinite groupoids, 
the following are known. 

\begin{theorem}[{\cite{M15crelle}}]\label{pi/thm}
Let $\G$ be a purely infinite groupoid. 
\begin{enumerate}
\item If $\G$ is minimal, 
then the commutator subgroup $D([[\G]])$ is simple. 
\item The index map $I:[[G]]\to H_1(\G)$ is surjective. 
\end{enumerate}
\end{theorem}

V. V. Nekrashevych \cite{Ne15} recently proved that 
if a purely infinite groupoid $\G$ is minimal and expansive, 
then $D([[\G]])$ is finitely generated 
(see \cite{Ne15} for the definition of expansive groupoids). 

It is not known 
if all minimal purely infinite groupoids satisfy the AH conjecture. 
Here, we describe a situation 
where the AH is inherited from a smaller groupoid to a larger groupoid. 

\begin{proposition}[{\cite[Theorem 4.4, Proposition 4.5]{M15}}]\label{TR}
Let $\G$ be a minimal \'etale groupoid. 
Let $c:\G\to\Z$ be a continuous surjective homomorphism and 
let $\H=\Ker c$. 
Assume either of the following conditions. 
\begin{enumerate}
\item $\H$ is a principal, minimal, almost finite groupoid 
with $M(\H)=\{\mu\}$, 
and there exists a real number $0<\lambda<1$ such that, 
for any compact open $\G$-set $U\subset c^{-1}(1)$, 
$\mu(r(U))=\lambda\mu(s(U))$ holds. 
\item $\H$ is a minimal, purely infinite groupoid 
satisfying the AH conjecture. 
\end{enumerate}
Then, $\G$ is purely infinite and satisfies the AH conjecture. 
\end{proposition}

\begin{example}\label{SFT/AH}
Let $\G_A$ be an SFT groupoid, 
where $A$ is the adjacency matrix of 
an irreducible finite directed graph $(V,E)$ (see Example \ref{SFT/def}). 
There exists a topologically mixing one-sided SFT $(X_B,\sigma_B)$ 
such that $\G_A\cong\G_B$ (\cite[Lemma 5.6]{M15}). 
Define $c:\G_B\to\Z$ by $c(x,n,y)=n$. 
Then $c$ is a continuous surjective homomorphism 
and $\H=\Ker c$ is an AF groupoid. 
Since $(X_B,\sigma_B)$ is topologically mixing 
(or equivalently the matrix $B$ is primitive), 
$\H$ is minimal and $M(\H)$ is a singleton. 
One can check that condition (1) of the proposition above is satisfied 
(the real number $\lambda$ is the inverse of the Perron eigenvalue of $B$). 
Consequently, $\G_B$ satisfies the AH conjecture, and so does $\G_A$. 
(Indeed, we know that SFT groupoids have the strong AH property, 
see Example \ref{ExofAH} (3)). 
\end{example}

By the same technique as the example above, 
we can prove the following. 

\begin{theorem}[{\cite[Theorem 5.8]{M15}}]\label{prodSFT/AH}
Let $\G=\G_{A_1}\times\G_{A_2}\times\dots\times\G_{A_n}$ be 
a product groupoid of finitely many SFT groupoids. 
Then, $\G$ satisfies the AH conjecture (Conjecture \ref{AH}). 
\end{theorem}

\section{Various examples}

\subsection{Minimal $\Z$-actions}

In this subsection, 
we would like to review the results about \'etale groupoids 
of minimal $\Z$-actions on Cantor sets. 
Recall that these groupoids are almost finite 
(Proposition \ref{Z^N>almstfnt}). 
We identify a $\Z$-action $\Z\curvearrowright X$ 
with a homeomorphism on $X$. 

\begin{theorem}[{\cite[Theorem 2.4]{GPS95crelle}}]\label{Z/flip}
For $i=1,2$, let $\phi_i$ be a minimal homeomorphism 
on a Cantor set $X_i$. 
The following are equivalent. 
\begin{enumerate}
\item The \'etale groupoids $\G_{\phi_1}$ and $\G_{\phi_2}$ are 
isomorphic to each other. 
\item $\phi_1$ is flip conjugate to $\phi_2$. 
\end{enumerate}
\end{theorem}
\begin{proof}
We present a sketchy proof. (2)$\Rightarrow$(1) is obvious. 
Suppose that $\pi:\G_{\phi_1}\to\G_{\phi_2}$ is an isomorphism. 
We may assume $X=X_1=X_2$ and $\pi|X$ is the identity. 
Hence there exists a continuous map $c:\Z\times X\to\Z$ 
such that $\pi(n,x)=(c(n,x),x)$. 
One has $\phi_1(x)=\phi_2^{c(1,x)}(x)$ for all $x\in X$. 
Let $C=\max\{\lvert c(1,x)\rvert\mid x\in X\}<\infty$. 
It is easy to check that 
$c(n{+}m,x)=c(n,\phi_1^m(x))+c(m,x)$ holds for all $m,n\in\Z$ and $x\in X$. 
For each $x\in X$, 
the map $m\mapsto c(m,x)$ is a bijection on $\Z$, 
and $\lvert c(m{+}1,x)-c(m,x)\rvert\leq C$. 
Therefore, we have two possibilities: 
$c(m,x)\to+\infty$ as $m\to+\infty$ or $c(m,x)\to-\infty$ as $m\to+\infty$. 
Without loss of generality, we may assume that the former holds. 
Define $b:X\to\Z$ by 
\[
b(x)=\lim_{N\to\infty}\#\{n\in\N\mid c(n,x)\leq N\}-N. 
\]
Indeed, there exists $M\in\N$ such that if $c(n,x)>M$ then $n>0$. 
So, $b(x)=\#\{n\in\N\mid c(n,x)\leq N\}-N$ if $N\geq M$. 
Especially, $b$ is continuous. 
Define a continuous map $\gamma:X\to X$ by $\gamma(x)=\phi_2^{-b(x)}(x)$. 
For sufficiently large $N$, we can verify 
\begin{align*}
b(\phi_1(x))&=\#\{n\in\N\mid c(n,\phi_1(x))\leq N\}-N\\
&=\#\{n\in\N\mid c(n{+}1,x)\leq N+c(1,x)\}-N\\
&=\#\{n\in\N\mid c(n,x)\leq N+c(1,x)\}-1-N\\
&=\#\{n\in\N\mid c(n,x)\leq N+c(1,x)\}-(N+c(1,x))+c(1,x)-1\\
&=b(x)+c(1,x)-1, 
\end{align*}
which implies $\phi_2\circ\gamma=\gamma\circ\phi_1$. 
Since every orbit of $\phi_1$ (or $\phi_2$) is infinite, 
one can show that $\gamma$ is a homeomorphism. 
It follows that $\phi_1$ is conjugate to $\phi_2$. 
\end{proof}

\begin{theorem}
Let $\phi$ be a minimal homeomorphism on a Cantor set $X$. 
\begin{enumerate}
\item $H_0(\G_\phi)$ is isomorphic to 
$C(X,\Z)/\{f-f\circ\phi\mid f\in C(X,\Z)\}$, 
$H_1(\G_\phi)\cong\Z$ and $H_n(\G_\phi)=0$ for $n\geq2$. 
\item $D([[\G_\phi]])$ is simple. 
\item The index map $I:[[\G_\phi]]\to H_1(\G_\phi)$ is surjective. 
\item $[[\G_\phi]]_0/D([[\G_\phi]])$ is isomorphic to 
$H_0(\G_\phi)\otimes\Z_2$. 
\end{enumerate}
\end{theorem}
\begin{proof}
(1) is obvious because $H_n(\G_\phi)$ is isomorphic to 
the group homology $H_n(\Z,C(X,\Z))$ (see Example \ref{ExofH} (2)). 

(2) is a special case of Theorem \ref{almstfnt/thm} (1). 

(3) is clear because of $I(\phi)=1\in\Z=H_1(\G_\phi)$. 
(It can be also viewed as a special case of Theorem \ref{almstfnt/thm} (2)). 

We give a brief explanation of (4). 
See \cite{M06IJM} for a detailed proof. 
Fix a point $y\in X$. 
Define a subgroupoid $\H\subset\G_\phi=\Z\times X$ by 
\[
\H=\G_\phi\setminus\{(n,\phi^m(y))\mid m\leq0<n{+}m
\text{ \ or \ }n{+}m\leq 0<m\}. 
\]
Then $\H$ is open and 
becomes a minimal AF subgroupoid with the relative topology 
(\cite[Theorem 4.3]{GPS04ETDS}). 
Evidently $[[\H]]$ is a subgroup of $[[\G_\phi]]$. 
As observed in Example \ref{AF/tfg}, 
the topological full group $[[\H]]$ of the AF groupoid $\H$ is 
an increasing union of subgroups isomorphic to 
finite direct sums of symmetric groups, 
and the inclusion map of each step is given by the edge set $E_n$. 
It follows that $D([[\H]])$ is an increasing union of subgroups 
isomorphic to finite direct sums of alternating groups. 
Since $\H$ is minimal, 
we can easily deduce that $D([[\H]])$ is simple. 
From this, with some extra work, 
we get the simplicity of $D([[\G_\phi]])$. 
On the other hand, $[[\H]]/D([[\H]])$ is isomorphic to 
an inductive limit of finite direct sums of $\Z_2$, 
and the connecting map of each step is given by the edge set $E_n$. 
Hence one has $[[\H]]/D([[\H]])\cong H_0(\H)\otimes\Z_2$. 
The inclusion map $\H\to\G$ induces $H_0(\H)\cong H_0(\G_\phi)$, 
which implies $[[\H]]/D([[\H]])\cong H_0(\G_\phi)\otimes\Z_2$.
The index map $I$ kills all the elements of finite order, 
and so $[[\H]]$ is contained in $[[\G_\phi]]_0$. 
Then, with some extra work, 
we can show $[[\G_\phi]]_0/D([[\G_\phi]])\cong H_0(\G_\phi)\otimes\Z_2$.
\end{proof}

For $\alpha\in[[\G_\phi]]$, $I(\alpha)\in\Z$ is computed as follows 
(see \cite[Section 5]{GPS99Israel} for the detailed argument). 
Fix a $\phi$-invariant probability measure $\mu\in M(\G_\phi)$. 
By Example \ref{transf/tfg}, there exists a continuous map $n:X\to\Z$ 
such that $\alpha(x)=\phi^{n(x)}(x)$ for all $x\in X$. 
Define $I':[[\G_\phi]]\to\R$ by 
\[
I'(\alpha)=\int_Xn(x)\,d\mu(x). 
\]
It is easy to see that $I'$ is a homomorphism and $I'(\phi)=1$. 
By \cite[Lemma 5.3]{GPS99Israel} (or \cite[Lemma 4.1]{M06IJM}), 
the group $[[\G_\phi]]_0$ is generated by elements of finite order. 
Therefore the kernel of $I'$ contains $[[\G_\phi]]_0$. 
So, we can conclude $I=I'$. 

We say that a homeomorphism $\phi$ on a Cantor set is expansive 
if there exists a continuous map $f$ from $X$ to a finite set $A$ 
such that $f(\phi^n(x))=f(\phi^n(y))$ for all $n\in\Z$ implies $x=y$. 

\begin{theorem}[{\cite[Theorem 5.4, Theorem 5.7]{M06IJM}}]\label{Z/finite}
Let $\phi$ be a minimal homeomorphism on a Cantor set $X$. 
\begin{enumerate}
\item $D([[\G_\phi]])$ is finitely generated 
if and only if $\phi$ is expansive. 
\item $D([[\G_\phi]])$ never be finitely presented. 
\end{enumerate}
\end{theorem}

Notice that (1) is a special case of \cite[Theorem 5.6]{Ne15}. 
In particular, the same statement holds 
for any free minimal action $\phi:\Z^N\curvearrowright X$. 
(2) was shown in \cite{M06IJM} 
by expressing $(X,\phi)$ as a decreasing `intersection' of 
two-sided shifts of finite type. 
Later, R. Grigorchuk and K. Medynets \cite{GM14Sbornik} proved that 
$[[\G_\phi]]$ is locally embeddable into finite groups, 
which implies that $D([[\G_\phi]])$ never be finitely presented. 
We do not know if $D([[\G_\phi]])$ can be finitely presented 
when $\phi$ is a free minimal action of $\Z^N$. 

K. Juschenko and N. Monod obtained the following remarkable result. 

\begin{theorem}[{\cite[Theorem A]{JM13Ann}}]\label{Z/amenable}
Let $\phi$ be a minimal homeomorphism on a Cantor set $X$. 
Then $[[\G_\phi]]$ is amenable. 
\end{theorem}

In the proof of this theorem, 
the notion of extensive amenability plays the central role. 
This property was first introduced (without a name) in \cite{JM13Ann}, 
and studied further in \cite{JNS13,JBMS15}. 

We recall the definition of extensive amenability 
from \cite[Definition 1.1]{JBMS15}. 
Let $\alpha:G \curvearrowright Z$ be 
an action of a discrete group $G$ on a set $Z$. 
We denote by $\mathcal{P}_f(Z)$ the set of all finite subsets of $Z$. 
The collection $\mathcal{P}_f(Z)$ is an abelian group 
for the operation $\Delta$ of symmetric difference. 
The action $\alpha$ naturally extends to 
$\alpha:G\curvearrowright\mathcal{P}_f(Z)$. 
We say that $\alpha:G\curvearrowright Z$ is extensively amenable 
if there exists a $G$-invariant mean 
(i.e. finitely additive probability measure) $m$ on $\mathcal{P}_f(Z)$ 
such that $m(\{F\in\mathcal{P}_f(Z)\mid E\subset F\})=1$ 
for any $E\in\mathcal{P}_f(Z)$. 
In \cite[Lemma 3.1]{JM13Ann}, it was shown that 
$\alpha:G\curvearrowright Z$ is extensively amenable if and only if 
the action of $\mathcal{P}_f(Z)\rtimes G$ on $\mathcal{P}_f(Z)$ 
admits an invariant mean. 

We denote by $W(\Z)$ the group of all permutations $g$ of $\Z$ 
for which the quantity $\sup\{\lvert g(j)-j\rvert\mid j\in\Z\}$ is finite. 
In \cite[Theorem C]{JM13Ann}, it was shown that 
the natural action $W(\Z)\curvearrowright\Z$ is extensively amenable. 
(This part is technically quite hard.) 
It follows that 
the action of $\mathcal{P}_f(\Z)\rtimes W(\Z)$ on $\mathcal{P}_f(\Z)$ 
admits an invariant mean. 

Let $\phi$ be a minimal homeomorphism on a Cantor set $X$. 
We would like to show that $[[\G_\phi]]$ is amenable. 
Fix a point $x\in X$. 
We can define a map $\pi:[[\G_\phi]]\to W(\Z)$ 
so that $\gamma(\phi^j(x))=\phi^{\pi(\gamma)(j)}(x)$ 
for every $\gamma\in[[\G_\phi]]$ and $j\in\Z$. 
The map $\pi$ is an injective homomorphism. 
Define a map $\tilde\pi:[[\G_\phi]]\to\mathcal{P}_f(\Z)\rtimes W(\Z)$ 
by $\tilde\pi(\gamma)=(\N\Delta\pi(\gamma)(\N),\pi(\gamma))$ 
for $\gamma\in[[\G_\phi]]$. 
It is routine to check that $\tilde\pi$ is an injective homomorphism. 
Since the action of $\mathcal{P}_f(\Z)\rtimes W(\Z)$ on $\mathcal{P}_f(\Z)$ 
admits an invariant mean, in order to show the amenability of $[[\G_\phi]]$, 
it suffices to prove that 
the stabiliser in $\tilde\pi([[\G_\phi]])$ of $E$ is amenable 
for any $E\in\mathcal{P}_f(\Z)$. 
By \cite[Lemma 4.1]{JM13Ann}, the stabiliser is locally finite, 
and hence amenable. 
This completes the proof. 

In \cite{JNS13,JBMS15}, 
the notion of extensive amenability is used 
to prove amenability of various kinds of groups. 
Among others, 
it was shown that all subgroups of the group of 
interval exchange transformations that have angular components 
of rational rank $\leq2$ are amenable (\cite[Theorem 5.1]{JBMS15}). 
In particular, 
when $\phi:\Z^2\curvearrowright X$ is a free minimal action 
arising from two irrational rotations on the circle 
(see \cite[Example 30]{GPS04Abel}), 
the topological full group $[[\G_\phi]]$ is amenable. 
On the other hand, it is known that 
there exists a free minimal action $\phi:\Z^2\curvearrowright X$ 
on a Cantor set 
such that $[[\G_\phi]]$ contains the non-abelian free group 
(\cite{EM13PAMS}). 
It may be a rather complicated problem to determine 
when $[[\G_\phi]]$ is amenable for $\phi:\Z^2\curvearrowright X$.

\subsection{Shifts of finite type}

In this subsection, 
we would like to review the results about \'etale groupoids 
of one-sided shifts of finite type (see Definition \ref{SFT/def}). 
Note that these groupoids are purely infinite and minimal 
(\cite[Lemma 6.1]{M15crelle}). 

\begin{theorem}[{\cite[Theorem 3.6]{MM14Kyoto}}]
Let $(X_A,\sigma_A)$ and $(X_B,\sigma_B)$ be 
two irreducible one-sided shifts of finite type. 
The following conditions are equivalent. 
\begin{enumerate}
\item $(X_A,\sigma_A)$ and $(X_B,\sigma_B)$ are 
continuously orbit equivalent. 
\item The \'etale groupoids $\G_A$ and $\G_B$ are isomorphic. 
\item There exists an isomorphism $\phi:C^*_r(\G_A)\to C^*_r(\G_B)$ 
such that $\phi(C(X_A))=C(X_B)$. 
\item There exists an isomorphism $\pi:H_0(\G_A)\to H_0(\G_B)$ 
such that $\pi([1_{X_A}])=[1_{X_B}]$ and 
$\sgn(\det(\id-A))=\sgn(\det(\id-B))$. 
\end{enumerate}
\end{theorem}

For the definition of continuous orbit equivalence, 
see \cite[Section 2.1]{MM14Kyoto}. 
As mentioned in Example \ref{ExofH} (3), we have 
\[
H_n(\G_A)\cong\begin{cases}\Coker(\id-A^t)&n=0\\
\Ker(\id-A^t)&n=1\\0&n\geq2. \end{cases}
\]
The element $[1_{X_A}]\in H_0(\G_A)$ corresponds to 
the equivalence class of $(1,1,\dots,1)$ in $\Coker(\id-A^t)$. 

\begin{proof}
We present a sketchy proof of the theorem above. 
The equivalence (1)$\iff$(2) directly follows from the definition. 
The equivalence (2)$\iff$(3) is a special case of Theorem \ref{Renault}. 

(4)$\Rightarrow$(3). 
Let $(\bar{X}_A,\bar{\sigma}_A)$ and $(\bar{X}_B,\bar{\sigma}_B)$ 
denote the two-sided shifts of finite type. 
By the result of J. Franks \cite{F84ETDS}, 
$H_0(\G_A)\cong H_0(\G_B)$ and $\sgn(\det(\id-A))=\sgn(\det(\id-B))$ 
imply that $(\bar{X}_A,\bar{\sigma}_A)$ and $(\bar{X}_B,\bar{\sigma}_B)$ 
are flow equivalent. 
It follows from \cite[Theorem 4.1]{CK80Invent} that 
there exists an isomorphism 
$\phi:C^*_r(\G_A)\otimes\K\to C^*_r(\G_B)\otimes\K$ 
such that $\phi(C(X_A)\otimes\mathcal{C})=C(X_B)\otimes\mathcal{C}$, 
where $\mathcal{C}\cong c_0(\Z)$ is the maximal abelian subalgebra of $\K$ 
consisting of diagonal operators. 
Since the isomorphism $\pi:H_0(\G_A)\to H_0(\G_B)$ carries 
$[1_{X_A}]$ to $[1_{X_B}]$, 
a suitable modification of $\phi$ yields the desired isomorphism. 

(2)$\Rightarrow$(4). 
Clearly, 
$\G_A\cong\G_B$ implies $(H_0(\G_A),[1_{X_A}])\cong(H_0(\G_B),[1_{X_B}])$. 
The ordered cohomology group of $(\bar{X}_A,\bar{\sigma}_A)$, 
introduced by M. Boyle and D. Handelman \cite{BH96Israel}, is 
the abelian group 
$\bar{H}^A=C(\bar X_A,\Z)
/\{\xi-\xi\circ\bar\sigma_A\mid\xi\in C(\bar X_A,\Z)\}$ 
with the positive cone $\bar H^A_+=\{[\xi]\in\bar H^A\mid\xi\geq0\}$. 
We can prove that 
$\G_A\cong\G_B$ implies $(\bar H^A,\bar H^A_+)\cong(\bar H^B,\bar H^B_+)$ 
(see \cite[Theorem 3.5]{MM14Kyoto}). 
Then, by \cite[Theorem 1.12]{BH96Israel}, 
$(\bar{X}_A,\bar{\sigma}_A)$ and $(\bar{X}_B,\bar{\sigma}_B)$ 
are flow equivalent. 
As a result, we obtain $\sgn(\det(\id-A))=\sgn(\det(\id-B))$. 
\end{proof}

\begin{theorem}\label{SFT/thm}
Let $\G_A$ be an SFT groupoid. 
\begin{enumerate}
\item $D([[\G_A]])$ is simple. 
\item The index map $I:[[\G_A]]\to H_1(\G_A)$ is surjective. 
\item $[[\G_A]]_0/D([[\G_A]])$ is isomorphic to $H_0(\G_A)\otimes\Z_2$. 
\end{enumerate}
\end{theorem}
\begin{proof}
(1) and (2) immediately follow from Theorem \ref{pi/thm} (1) and (2). 

(3). 
By Example \ref{SFT/AH}, 
\[
\begin{CD}
H_0(\G_A)\otimes\Z_2@>j>>[[\G_A]]_\ab@>>>H_1(\G_A)@>>>0
\end{CD}
\]
is exact. 
It suffices to show that $j$ is injective and has a left inverse. 
This was shown in \cite[Section 6.6]{M15crelle}, 
by using a finite presentation of $[[\G_A]]$. 
Here, we would like to describe another approach. 

As mentioned in Example \ref{SFT/tfg}, 
when $A$ is a $1\times1$ matrix $[n]$, 
$[[\G_{[n]}]]$ is the Higman-Thompson group $V_{n,1}$. 
It is well-known that 
the abelianization of $V_{n,1}$ is trivial if $n$ is even, 
and is $\Z_2$ if $n$ is odd. 
Suppose that an SFT groupoid $\G_A$ is given. 
Let $\phi:H_0(\G_A)\to\Z_2\cong H_0(\G_{[3]})$ be a homomorphism. 
Choose a nonempty clopen set $Y\subset X_{[3]}$ 
so that $[1_Y]=\phi([1_{X_A}])$. 
Set $\H=\G_{[3]}|Y$. 
We have $[[\H]]_\ab\cong\Z_2$. 
By Theorem \ref{embedding}, 
we can find an embedding $\pi:C^*_r(\G_A)\to C^*_r(\H)$ 
which induces 
\[
\begin{CD}
1@>>>U(C(\G_A^{(0)}))@>>>
N(C(X_A),C^*_r(\G_A))@>>>[[\G_A]]@>>>1\\
@.@V\pi VV@V\pi VV@VVV@.\\
1@>>>U(C(Y))@>>>
N(C(Y),C^*_r(\H))@>>>[[\H]]@>>>1, 
\end{CD}
\]
and $\phi([1_P])=[\pi(1_P)]$ for any clopen set $P\subset X_A$. 
Let us denote the embedding $[[\G_A]]\to[[\H]]$ 
(and also the induced homomorphism $[[\G_A]]_\ab\to[[\H]]_\ab$) by $\pi_*$. 
Let $U$ be a compact open $\G_A$-set 
satisfying $r(U)\cap s(U)=\emptyset$. 
Define $\tau_U\in[[\G_A]]$ by 
\[
\tau_U(x)=\begin{cases}\theta(U)(x)&x\in s(U)\\
\theta(U^{-1})(x)&x\in r(U)\\x&\text{otherwise. }\end{cases}
\]
Then $j([1_{s(U)}]\otimes\bar{1})$ equals 
the equivalence class of $\tau_U$ in $[[\G_A]]_\ab$. 
It is easy to see that 
the equivalence class of $\pi_*(\tau_U)$ equals 
$j(\phi([1_{s(U)}])\otimes\bar{1})$. 
Thus, 
$\pi_*(j([1_{s(U)}]\otimes\bar{1}))=j(\phi([1_{s(U)}])\otimes\bar{1})$. 
Hence $\pi_*\circ j=j\circ(\phi\otimes\id)$. 
Since 
the homomorphism $\phi:H_0(\G_A)\to\Z_2\cong H_0(\G_{[3]})$ was arbitrary, 
we obtain the desired conclusion. 
\end{proof}

In the proof above, we use the following embedding theorem. 

\begin{theorem}[{\cite[Proposition 5.14]{M15}}]\label{embedding}
Let $\G_A$ be an SFT groupoid and 
let $\H$ be a minimal, purely infinite \'etale groupoid. 
Suppose that $\phi:H_0(\G_A)\to H_0(\H)$ is a homomorphism 
satisfying $\phi([1_{X_A}])=[1_{\H^{(0)}}]$. 
Then there exists a unital homomorphism 
$\pi:C^*_r(\G_A)\to C^*_r(\H)$ such that the following hold. 
\begin{enumerate}
\item $\pi(C(X_A))\subset C(\H^{(0)})$. 
\item For any compact open $\G_A$-set $U$, 
there exists a compact open $\H$-set $V$ 
such that $\pi(1_U)=1_V$. 
\item For any clopen set $P\subset X_A$, 
$[\pi(1_P)]=\phi([1_P])$ in $H_0(\H)$. 
\end{enumerate}
In particular, 
$\pi$ induces an embedding of $[[\G_A]]$ into $[[\H]]$. 
\end{theorem}

The proof of this theorem uses the fact that 
the Cuntz-Krieger algebra $C^*_r(\G_A)$ is characterized 
as the universal $C^*$-algebra generated by partial isometries 
subject to the Cuntz-Krieger relations (\cite{CK80Invent}). 

As for finiteness condition, the following is known. 

\begin{theorem}[{\cite[Section 6]{M15crelle}}]\label{SFT/finite}
Let $\G_A$ be an SFT groupoid. 
\begin{enumerate}
\item $[[\G_A]]$ is of type F$_\infty$. 
(In particular, it is finitely presented.) 
\item $[[\G_A]]_0$ and $D([[\G_A]])$ are finitely generated. 
\end{enumerate}
\end{theorem}

For the Higman-Thompson groups $V_{n,r}$ (and also $F_{n,r}$, $T_{n,r}$), 
K. S. Brown \cite{Br87JPAA} proved that they are of type F$_\infty$. 
The theorem above is a generalization of this result, 
and its proof uses Brown's criterion (\cite[Corollary 3.3]{Br87JPAA}). 

As mentioned in Section 7, 
$[[\G_A]]$ is not amenable, because $\G_A$ is purely infinite. 
But, $[[\G_A]]$ has a weaker version of the amenability. 

\begin{theorem}[{\cite[Theorem 6.7]{M15crelle}}]\label{SFT/Haagerup}
Let $\G_A$ be an SFT groupoid. 
The topological full group $[[\G_A]]$ has the Haagerup property. 
\end{theorem}

D. S. Farley \cite{F03IMRN} proved that 
the Higman-Thompson group $V_{2,1}$ has the Haagerup property. 
In \cite{M15crelle}, the Haagerup property of $[[\G_A]]$ was shown 
by modifying the argument of \cite{F03IMRN}. 
But, making use of Theorem \ref{embedding}, 
we can embed $[[\G_A]]$ to $[[\G_{[2]}]]\cong V_{2,1}$, 
and so the Haagerup property of $[[\G_A]]$ follows immediately. 

Soficity and exactness of the Higman-Thompson groups (or $[[\G_A]]$) 
is an important open problem.

\subsection{Products of shifts of finite type}

In this subsection, 
we would like to review the results about product groupoids 
$\G=\G_{A_1}\times\G_{A_2}\times\dots\times\G_{A_n}$ 
of SFT groupoids. 
Evidently, these groupoids are purely infinite and minimal. 
(In general, if $\G$ is purely infinite, 
then $\G\times\H$ is purely infinite, too.) 
The groupoid $C^*$-algebra $C^*_r(\G)$ is isomorphic to 
the tensor product 
$C^*_r(\G_{A_1})\otimes C^*_r(\G_{A_2})\otimes\dots\otimes C^*_r(\G_{A_n})$. 

In what follows, for an irreducible one-sided SFT $(X_A,\sigma_A)$, 
the equivalence class of $1_{X_A}$ in $H_0(\G_A)$ is 
denoted by $u_A$. 

\begin{theorem}[{\cite[Theorem 5.12]{M15}}]\label{prodSFT1}
Let $\G=\G_{A_1}\times\G_{A_2}\times\dots\times\G_{A_m}$ and 
$\H=\G_{B_1}\times\G_{B_2}\times\dots\times\G_{B_n}$ be 
product groupoids of SFT groupoids. 
Then $\G\cong\H$ if and only if the following are satisfied. 
\begin{enumerate}
\item $m=n$. 
\item There exist a permutation $\sigma$ of $\{1,2,\dots,n\}$ and 
isomorphisms $\phi_i:H_0(\G_{A_i})\to H_0(\G_{B_{\sigma(i)}})$ 
such that $\det(\id-A_i)=\det(\id-B_{\sigma(i)})$ and 
\[
(\phi_1\otimes\phi_2\otimes\dots\otimes\phi_n)
(u_{A_1}\otimes u_{A_2}\otimes\dots\otimes u_{A_n})
=u_{B_{\sigma(1)}}\otimes u_{B_{\sigma(2)}}\otimes\dots
\otimes u_{B_{\sigma(n)}}. 
\]
In particular, 
$\G_{A_i}$ and $\G_{B_{\sigma(i)}}$ are Morita equivalent. 
\end{enumerate}
\end{theorem}

\begin{theorem}[{\cite{M15}}]\label{prodSFT2}
Let $\G=\G_{A_1}\times\G_{A_2}\times\dots\times\G_{A_n}$ be 
a product groupoid of SFT groupoids. 
\begin{enumerate}
\item $H_k(\G)$ is isomorphic to 
\begin{align*}
&\left(\Z^{\binom{n-1}{k}}
\otimes H_0(\G_{A_1})\otimes H_0(\G_{A_2})\otimes
\dots\otimes H_0(\G_{A_n})\right)\\
&\oplus\left(\Z^{\binom{n-1}{k-1}}
\otimes H_1(\G_{A_1})\otimes H_1(\G_{A_2})\otimes
\dots\otimes H_1(\G_{A_n})\right), 
\end{align*}
where $\binom{n}{k}$ denote the binomial coefficients and 
they are understood as zero unless $0\leq k\leq n$. 
The equivalence class of the constant function $1_{\G^{(0)}}$ 
in $H_0(\G)=H_0(\G_{A_1})\otimes\dots\otimes H_0(\G_{A_n})$ 
is $u_{A_1}\otimes\dots\otimes u_{A_n}$. 
\item $\G$ satisfies the HK conjecture. 
\item $D([[\G]])$ is simple. 
\item The index map $I:[[\G]]\to H_1(\G)$ is surjective. 
\item $\G$ satisfies the AH conjecture. 
\end{enumerate}
\end{theorem}
\begin{proof}
(1) is obtained from the K\"unneth theorem (Theorem \ref{Kunneth}). 

(2) is an immediate consequence of Theorem \ref{Kunneth} and 
the K\"unneth theorem for $K$-groups of $C^*$-algebras. 

(3) and (4) readily follow from Theorem \ref{pi/thm} (1) and (2). 

(5) is shown by an inductive application of Proposition \ref{TR}. 
See also Example \ref{SFT/AH} and Theorem \ref{prodSFT/AH}. 
\end{proof}

The topological full group of 
$\G=\G_{A_1}\times\G_{A_2}\times\dots\times\G_{A_n}$ is 
a generalization of the higher dimensional Thompson groups $nV_{k,r}$. 
M. G. Brin introduced 
the notion of higher dimensional Thompson groups $nV_{k,r}$ 
in \cite[Section 4.2]{Brin04GeomD}. 
These groups can be considered as an $n$-dimensional analogue of 
the Higman-Thompson group $V_{k,r}=1V_{k,r}$. 
Brin proved that $V_{k,r}$ and $2V_{2,1}$ are not isomorphic, 
and that $2V_{2,1}$ is finitely presented. 
He also proved that $nV_{2,1}$ is simple for all $n\in\N$ 
in \cite{Brin10PublMat}. 
W. Dicks and C. Mart\'inez-P\'erez \cite{DMP14Israel} 
proved that $nV_{k,r}\cong n'V_{k',r'}$ if and only if 
$n=n'$, $k=k'$ and $\gcd(k{-}1,r)=\gcd(k'{-}1,r')$. 

Define an $r\times r$ matrix $A_{k,r}$ by 
\[
A_{k,r}=\begin{bmatrix}0&0&\ldots&0&k\\1&0&\ldots&0&0\\0&1&\ldots&0&0\\
\vdots&\vdots&\ddots&\vdots&\vdots\\0&0&\ldots&1&0\end{bmatrix}. 
\]
The topological full group $[[\G_{A_{k,r}}]]$ of the SFT groupoid $A_{k,r}$ 
is naturally isomorphic to the Higman-Thompson group $V_{k,r}$ 
(see \cite[Section 6.7.1]{M15crelle}). 
By Example \ref{ExofH} (3), 
$H_0(\G_{A_{k,r}})\cong\Z_{k{-}1}$, $H_n(\G_{A_{k,r}})=0$ for $n\geq1$ and 
$u_{A_{k,r}}$ corresponds to $\bar{r}\in\Z_{k{-}1}$. 
It is not so hard to see that 
the higher dimensional Thompson group $nV_{k,r}$ is isomorphic to 
the topological full group $[[\G]]$ of 
the product groupoid 
\[
\G=\G_{A_{k,r}}
\times\overbrace{\G_{A_{k,1}}\times\dots\times\G_{A_{k,1}}}^{n-1}. 
\]
It follows from Theorem \ref{prodSFT2} (3) that 
the commutator subgroup $D([[\G]])$ is simple. 
By Theorem \ref{prodSFT2} (1), 
we get $H_l(\G)\cong(\Z_{k{-}1})^{\binom{n{-}1}{l}}$. 
Therefore, Theorem \ref{prodSFT2} (5) tells us that 
$[[\G]]$ is simple if and only if $k=2$. 
This reproves the result of Brin \cite{Brin10PublMat}. 
Also, by applying Theorem \ref{prodSFT1}, 
we get a new proof of the classification theorem 
by Dicks and Mart\'inez-P\'erez \cite{DMP14Israel}. 

As for the cohomological finiteness condition, 
C. Mart\'inez-P\'erez, F. Matucci and B. E. A. Nucinkis 
\cite{MPMN13} 
proved that $nV_{k,1}$ (and many other relatives) are of type F$_\infty$. 
We do not know if the same holds for topological full groups 
of product groupoids of SFT groupoids 
$\G=\G_{A_1}\times\G_{A_2}\times\dots\times\G_{A_n}$. 

In \cite[Section 5.5]{M15}, 
we completely determined the abelianization of the topological full group of 
$\G=\G_{A_1}\times\G_{A_2}\times\dots\times\G_{A_n}$, 
but here we do not state the result precisely 
because it would be quite complicated. 
Instead, 
let us consider a special case, namely products of one-sided full shifts. 
Let $k:\{1,2,\dots,n\}\to\N\setminus\{1\}$ be a map. 
Set 
\[
\G=\G_{[k(1)]}\times\G_{[k(2)]}\times\dots\times\dots\G_{[k(n)]}. 
\]
Let $g=\gcd\{k(i){-}1\mid i=1,2,\dots,n\}$. 
Then $H_0(\G)\cong\Z_g$ and $H_1(\G)\cong(\Z_g)^{n-1}$. 
By Theorem \ref{prodSFT2} (5), 
\[
\begin{CD}
\Z_g\otimes\Z_2@>j>>[[\G]]_\ab@>I>>(\Z_g)^{n-1}@>>>0
\end{CD}
\]
is exact. 

\begin{theorem}[{\cite[Theorem 5.23]{M15}}]\label{prodSFT3}
Let $\G$ be as above. 
\begin{enumerate}
\item If $k(i)$ is even for some $i$, then $[[\G]]_\ab\cong(\Z_g)^{n-1}$. 
\item If $k(i)$ is odd for all $i$ and $\#\{i\mid k(i)\in4\Z{+}3\}\leq1$, 
then $[[\G]]_\ab\cong\Z_2\oplus(\Z_g)^{n-1}$. 
\item If $k(i)$ is odd for all $i$ and $\#\{i\mid k(i)\in4\Z{+}3\}=2$, 
then $[[\G]]_\ab\cong\Z_{2g}\oplus(\Z_g)^{n-2}$. 
\item If $k(i)$ is odd for all $i$ and $\#\{i\mid k(i)\in4\Z{+}3\}\geq3$, 
then $[[\G]]_\ab\cong(\Z_g)^{n-1}$. 
In particular, $\G$ does not have the strong AH property. 
\end{enumerate}
\end{theorem}
\begin{proof}
(1). 
Since $g$ is odd, $\Z_g\otimes\Z_2=0$. 
So $[[\G]]_\ab\cong(\Z_g)^{n-1}$. 

(2). 
Let us consider the case that $k(i)\in4\Z{+}1$ for all $i=1,2,\dots,n$. 
Let $\H$ be the direct product of $n$ copies of $\G_{[5]}$. 
By \cite[Lemma 5.19 (1)]{M15}, 
there exists a homomorphism $\rho:[[\H]]_\ab\to\Z_2$ 
such that $\rho\circ j$ is nonzero. 
(In \cite[Lemma 5.19 (1)]{M15}, 
$[[\H]]$ is embedded into a group (named $W_{n,k}$), 
and its generators and relations are explicitly written down. 
Using them, one can obtain the homomorphism $\rho$ to $\Z_2$.) 
For each $i=1,2,\dots,n$, 
we define a homomorphism $\phi_i:H_0(\G_{[k(i)]})\to H_0(\G_{[5]})$ 
by $\phi_i(\bar{1})=\bar{1}$. 
Applying Theorem \ref{embedding} to each $\phi_i$, 
we get an embedding $\pi:[[\G]]\to[[\H]]$. 
In the same way as the proof of Theorem \ref{SFT/thm} (3), 
one can conclude that $\rho\circ\pi\circ j$ is nonzero. 
Thus, $[[\G]]_\ab\cong\Z_2\oplus(\Z_g)^{n-1}$. 

When $k(i)$ is odd for all $i$ and $\#\{i\mid k(i)\in4\Z{+}3\}=1$, 
the same argument works by using \cite[Lemma 5.19 (2)]{M15}. 

(3). 
Almost the same argument as above works, 
by using \cite[Lemma 5.19 (3)]{M15}. 
But, the range of the homomorphism $\rho$ becomes $\Z_4$. 
Thus, the map $j:H_0(\G)\otimes\Z_2\to[[\G]]_\ab$ is injective 
but does not have a right inverse. 
Hence, we can conclude $[[\G]]_\ab\cong\Z_{2g}\oplus(\Z_g)^{n-2}$. 

(4). 
For simplicity, we assume $n=3$ and $k(1)=k(2)=k(3)=3$. 
It suffices to show that $j:H_0(\G)\otimes\Z_2\to[[\G]]_\ab$ is zero. 

Let $X_{[3]}=\{0,1,2\}^\N$ and 
let $(X_{[3]},\sigma_{[3]})$ be the full shift over three symbols. 
We define the clopen set $C_i\subset X_{[3]}$ by 
\[
C_i=\{(x_n)_{n\in\N}\in X_{[3]}\mid x_1=i\}. 
\]
Define a compact open $\G_{[3]}$-set $U_i\subset\G_{[3]}$ by 
\[
U_i=\{(x,1,y)\in\G_{[3]}\mid x\in C_i,\ \sigma_{[3]}(x)=y\}. 
\]
Let $t\in[[\G]]_\ab$ be the image of 
the generator of $H_0(\G)\otimes\Z_2\cong\Z_2$. 
We would like to show $t=0$. 
Define $\beta_{12}\in[[\G]]$ by 
\[
\beta_{12}(x,y,z)=
\theta(U_i\times U_i^{-1}\times X_{[3]})(x,y,z)\
\quad\text{when $y\in C_i$}
\]
for $(x,y,z)\in\G^{(0)}=X_{[3]}\times X_{[3]}\times X_{[3]}$. 
The homeomorphism $\beta_{12}$ is the so-called baker's map 
acting on the first and second coordinates of 
$X_{[3]}\times X_{[3]}\times X_{[3]}$, 
and its index $I(\beta_{12})$ is nonzero in $H_1(\G)\cong\Z_2\oplus\Z_2$. 
By \cite[Lemma 5.21 (4)]{M15}, 
one sees $2[\beta_{12}]=t$ in $[[\G]]_\ab$. 
We can define $\beta_{23}\in[[\G]]$ in the same way by 
\[
\beta_{23}(x,y,z)=
\theta(X_{[3]}\times U_i\times U_i^{-1})(x,y,z)\
\quad\text{when $z\in C_i$}. 
\]
Again one has $2[\beta_{23}]=t$. 
It is easy to see that 
$\beta_{12}\beta_{23}$ is equal to the baker's map 
acting on the first and third coordinates of 
$X_{[3]}\times X_{[3]}\times X_{[3]}$. 
Therefore, we get $2[\beta_{12}\beta_{23}]=t$. 
Consequently, we obtain $2t=t$, thus $t=0$. 
\end{proof}

Little is known about analytic properties of $[[\G]]$. 
For example, it is natural to ask 
if the topological full group $[[\G]]$ of 
$\G=\G_{A_1}\times\G_{A_2}\times\dots\times\G_{A_n}$ 
has the Haagerup property or not.

\end{document}